\documentclass[a4paper,11pt]{article}

\usepackage[latin1]{inputenc}
\usepackage[english]{babel}
\usepackage{amsmath}
\usepackage{amsthm}
\usepackage{amssymb}
\usepackage{oldgerm}
\usepackage{amscd}
\usepackage[all]{xy}
\usepackage{hyperref}
\usepackage{geometry}
\usepackage{pb-diagram}
\usepackage{color}
\usepackage{mathtools}

\title{Good Reduction for Endomorphisms of the Projective Line in Terms of the Branch Locus}
\author{Jung Kyu CANCI
\\
\small{Universit\"{a}t Basel, Mathematisches Institut, Spiegelgasse $1$, CH-$4051$ Basel}\\
\small\texttt{jungkyu.canci@unibas.ch}}

\newtheorem{Def}{Definition}[section]

\newtheorem{Teo}{Theorem}[section]
\newtheorem{Lemma}{Lemma}[section]
\newtheorem{Cor}{Corollary}[section]

\newtheorem{Rem}{Remark}[section]

\newcommand{\Q}{\mathbb{Q}}

\newcommand{\N}{\mathbb{N}}
\newcommand{\Z}{\mathbb{Z}}

\newcommand{\SR}{\mathbb{P}_1}

\newcommand{\nc}{\newcommand}
\nc{\mco}{\mathcal{O}}
\nc{\mcr}{\mathcal{R}}
\nc{\bdm}{\begin{displaymath}}
\nc{\edm}{\end{displaymath}}
\nc{\beq}{\begin{equation}}
\nc{\eeq}{\end{equation}}
\nc{\PGL}{{\rm PGL}}
\nc{\PGLv}{{\rm PGL}(2,(O_K)_v)}

\begin{document}

\maketitle

\begin{abstract}Let $K$ be a number field and $v$ a non archimedean valuation on $K$. We say that an endomorphism $\Phi\colon \SR\to \SR$ has good reduction at $v$ if there exists a model $\Psi$ for $\Phi$ such that $\deg\Psi_v$, the degree of the reduction of $\Psi$ modulo $v$, equals $\deg\Psi$ and $\Psi_v$ is separable. We prove a criterion for good reduction that is the natural generalization of a result due to Zannier in \cite{Uz3}. Our result is in connection with other two notions of good reduction, the simple and the critically good reduction. The last part of our article is dedicated to prove a characterization of the maps  whose iterates, in a certain sense, preserve the critically good reduction.
\end{abstract}

\section{Introduction} 
Let $K$ be a number field and $v$ a non archimedean valuation on $K$. Denote by $K_0$ the residue field at $v$ and $p$ its characteristic. Any rational map $\Phi\colon \SR\to\SR$ defined over $K$ can be written with algebraic integral coefficients in $K$ whose reduction modulo $v$ defines a reduced map $\Phi_v\colon\SR\to\SR$, thus with coefficients in $K_0$. We say that two endomorphisms $\Phi,\Psi\colon \SR\to\SR$ over $K$ are \emph{equivalent} if there exists an automorphism $A\in \PGL_2(K)$ such that $\Phi=\Psi\circ A$. 
We shall take in consideration the following definition of good reduction for endomorphism of $\SR$: 
\begin{Def}\label{gr}
We say that an endomorphism $\Phi\colon \SR\to\SR$ defined over $K$ has \emph{good reduction} at $v$ if there exists a map $\Psi$, equivalent to $\Phi$, such that the reduction $\Psi_v$ satisfies $\deg \Psi=\deg \Psi_v$ and $\Psi_v\notin K_0(x^p)$. 
We say that a map $\Phi$ defined over $K$ has \emph{potential good reduction} at $v$, if $\Phi$ has good reduction over a finite extension of $K$.
\end{Def}
The condition $\Psi_v\notin K_0(x^p)$ is equivalent to say that $\Psi_v$ is separable.  As usual, for any field $F$, let us denote by $\overline{F}$ its algebraic closure. We shall denote by $\mcr_\Phi$ the set in $\SR (\overline{K})$ of ramification points of a map $\Phi$. Suppose that $v$ is extended in some way to the whole $\overline{K}$. With abuse of notation, we shall denote with $v$ also the extended valuation. Let $(\Phi(\mcr_\Phi))_v$ be the subset of $\SR(\overline{K}_0)$ obtained from $\Phi(\mcr_\Phi)$ by reduction of its points modulo $v$. In the present article, it will be crucial the following condition
\beq\label{crv}
\# \Phi(\mcr_\Phi)=\# (\Phi(\mcr_\Phi))_v,
\eeq 
which is equivalent to say that for any pair of distinct points $P,Q\in \Phi(\mcr_\Phi)$ their reduction modulo $v$ remain distinct. 
 As already remarked in \cite{SzT}, the condition (\ref{crv}) does not depend on the particular choice of the extension of $v$ to $\overline{K}$. Indeed for any finite extension of $F$ of $K$ the Galois group ${\rm Gal}(F/K)$ acts transitively on the valuations over $F$ that extend the valuation $v$ over $K$. 
 
Recall that giving an endomorphism of $\SR(K)$ is the same of giving an element of $K(x)$ (i.e. rational functions).

An aim of this article is to give a proof of the following result. 

\begin{Teo}\label{mainT}
Let $K$, $v$, $K_0$ and $p$ be as above. Let $\Phi$ be a rational function in $K(x)$ of degree $d\geq 2$. 
Suppose that condition (\ref{crv}) holds and $\Phi$ does not have potential good reduction at $v$. 
 Then the characteristic $p$ divides the order of the monodromy group and some nonzero integers of the form $\left(\sum_{P\in A}e_\Phi(P)-\sum_{P\in B}e_\Phi(P)\right)$, where $A\subset \Phi^{-1}(\lambda)$ and $B\subset \Phi^{-1}(\mu)$, for each pair of two different branch values $\lambda,\mu\in \Phi(\mcr_\Phi)$.
\end{Teo}

The statement of the above theorem is a generalization of \cite[Theorem 1]{Uz3} where Zannier considered some endomorphisms $\Phi\colon\SR\to\SR$, with $\Phi=F/G$ where $F,G\in K[x]$ and $\deg(F-G)$ is as small as possible, as predicted by Mason's \emph{abc} inequality. Those above covers are unramified outside $\{0,1,\infty\}$ (this remark is contained in \cite{Uz1}) and as showed in \cite{Uz2} the ramification over 1 is all concentrated at the point at infinity. An important tool to treat these type of problems is the Riemann's Existence Theorem. It can be used in the more general setting of coverings between general curves, where one establishes a connection between the topological data of a covering and its field of definition. For example see in \cite{Schneps} the article by Birch where he gave an introduction in this topic and some references to earlier works due to Groethendieck and Fulton and later by Beckmann. The methods used by Fulton in \cite{Fulton} and by Beckmann in \cite{Beckmann} concern covers of general curves, but in our situation apply for primes that do not divide the order of the monodromy group. The Zannier's method, that we are going to generalize, uses completely new arguments and provides new sufficient conditions to have good reduction. The above cited results are not only for the case when $K$ is a number field. Also our arguments work in a more general setting, where essentially $K$ is a generic field equipped with a discrete valuation $v$ that  could be extended to $\overline{K}$ and condition (\ref{crv}) does not depend on the chosen extension of $v$. We have chosen to state our results in the case when $K$ is a number field, because it is the most important case for applications and to put our arguments in the setting used in some previous works, e.g. \cite{CanciPT} and \cite{SzT}.

The proof of Theorem \ref{mainT} is divided in two parts. We shall take a map $\Phi$ defined over $K$ satisfying condition (\ref{crv}). We shall assume that the characteristic $p$ does not divide the order of monodromy group or there exist  two different branch values $\lambda,\mu\in \Phi(\mcr_\Phi)$ such that $p$ does not divide each nonzero integers of the form $\left(\sum_{P\in A}e_\Phi(P)-\sum_{P\in B}e_\Phi(P)\right)$ where $A\subset \Phi^{-1}(\lambda), B\subset \Phi^{-1}(\mu)$. We shall prove that $\Phi$ has potential good reduction at $v$. The first part of our proof contains a generalization of Zannier's techniques used in \cite{Uz3} in order to see that for each rational function $\Phi$ described as above, there exists a map $\Psi$ equivalent to $\Phi$, such that the  points in the ramification fibers of $\Psi$ are $v$--integers and the reduction $\Psi_v$ is separable. The second part is essentially the new part of our proof. In this part we use the characterization of the Gauss norm on $K(x)$, with respect $v$ and $x$. This characterization affirms that the (Gauss) valuation associated to the Gauss norm is the unique one that extends $v$ to $K(x)$ such that the reduction $\bar{x}$, of $x$ modulo the Gauss valuation, is transcendental over $K_0$. More concretely, the residue field of $K(x)$, with respect the Gauss valuation, is $K_0(\bar{x})$ that is the field of fraction of the polynomials in $K_0[\bar{x}]$, where $\bar{x}$ is the representative of $x$ and is transcendental over $K_0$.  Furthermore, it is crucial the characterization of the good reduction for a map $\Phi$ defined over $K$: $\Phi$ has good reduction at $v$ if and only if the Gauss norm on $K(\Phi)$ with respect $v$ and $\Phi$ extends in a unique way to $K(x)$ and the extension given by the reduced fields is separable; see \cite[p.99]{Uz3} for a proof of the equivalence. This characterization combined with some arguments in which we apply the Riemann Hurwitz Formula will produce the second part of the proof of Theorem \ref{mainT}. 
The hypothesis of Theorem \ref{mainT} about the divisibility concerning the characteristic $p$ of the residue field is essentially the same as in Zannier's hypothesis. Indeed in \cite[Theorem 1]{Uz3} the divisibility condition is given with the unique pair $\{\lambda,\mu\}=\{0,\infty\}$. In the setting of Theorem\ref{mainT}, since the condition (\ref{crv}) holds for $\Phi$, for two branch values $\lambda,\mu$, we can assume that $\{\lambda,\mu\}=\{0,\infty\}$. Indeed it is enough to replace $\Phi$ with $A\circ\Phi$, where $A$ is associated to a $v$--invertible matrix (a matrix with $v$--integral coefficients whose determinant is a $v$--unit), that sends $\lambda$ and $\mu$ to $0$ and $\infty$, respectively. Note that $\Phi$ has good reduction at $v$ if and only if $A\circ\Phi$ has good reduction at $v$.

 In the literature there are several notions of good reduction for endomorphisms of $\SR$. In this article we shall take in consideration two other ones in addition to the above Definition \ref{gr}:

\begin{Def}[Simple good reduction]\label{sgr} We say that an endomorphism $\Phi$ of $\SR$, defined over $K$, has simple good reduction at $v$ if the reduced map $\Phi_v$ has the same degree of $\Phi$.
\end{Def}  The definition of simple good reduction is almost the same as the one in Definition \ref{gr} but in the simple good reduction we do not allow to change the model of $\Phi$. Note that in Definition \ref{sgr} we are not asking that $\Phi_v$ has to be separable. The notion of simple good reduction was introduced by Morton and  Silverman in \cite{MS}.

\begin{Def}[Critically good reduction] We say that an endomorphism $\Phi$ of $\SR$, defined over $K$, has critically good reduction at $v$ if condition (\ref{crv}) is verified and the same holds for the set $\mcr_\Phi$ of ramification points; that is $\# \mcr_\Phi=\# (\mcr_\Phi)_v$, where as above the set $(\mcr_\Phi)_v$ is the subset of $\SR(K_0)$ obtained from $\mcr_\Phi$ by reduction of its points modulo $v$.
\end{Def}
According to the autor's knowledge this last notion of good reduction was introduced by Szpiro and Tucker in \cite{SzT}.

In \cite{CanciPT} the authors investigate the connections between the simple and the critically good reduction. They proved that given a rational function $\Phi$ defined over $K$ such that condition (\ref{crv}) holds and $\Phi_v$ is separable, then $\Phi$ has critically good reduction at $v$ if and only if $\Phi$ has simple good reduction at $v$. Theorem \ref{mainT} provides under condition (\ref{crv}) a connection between these three notions of good reduction. 

After the preparation of the article \cite{CanciPT}, in a private communication Szpiro asked for which rational maps the critically good reduction is preserved under iteration. Therefore we studied the maps having the property given in the following definition:

 \begin{Def} We say that an endomorphism $\Phi$ of the projective line, defined over $K$, is  \emph{finitely
critical} if there exists a finite set $S$ of
valuations  of $K$, containing all the archimedean ones, such that all the iterates $\Phi^n$, with $n\geq 1$,
have critically good reduction at each valuation outside $S$.
\end{Def}
Note that the condition to be finitely
critical is stable under conjugation by elements of ${\rm PGL}_2(K)$. Indeed 
let $\Phi$ be a finitely critical map with associated finite set
$S$ of non archimedean valuations of critically bad reduction of all its iterates and the archimedean ones. For each $f\in {\rm
PGL}_2(K)$ let $\Phi^f=f^{-1}\circ\Phi\circ f$. Let $T$ be
the set of valuations of $K$ that are of simple bad reduction for
$f$; then each iterate of $\Phi^f$ (which is of the form
$(\Phi^f)^n=f^{-1}\circ\Phi^n\circ f$) has critically good
reduction at each prime outside $S\cup T$. This elementary
observation is quite important, because we are considering a map
with all its iterates. Hence we are, in a sense, studying the
dynamic associated to a map. Therefore it is important to know that the notion of finitely
critical map is a good dynamical condition; indeed the conjugation by elements of ${\rm
PGL}_2(K)$ preserves the dynamics. 

In \cite{CanciPT} we presented the map $\Phi(x)=x^2-x$ as an example of non finitely critical map. In fact $1/2$ is a critical point of $\Phi$ and $1/2$ is not a preperiodic point for $\Phi$, therefore its orbit $O_\Phi(1/2)$ is an infinite set of rational points. Thus, the set of primes of bad reduction for some iterate of $\Phi$ can not be finite. This example and studies about some particular families of rational maps (which we will present in Section \ref{cgr} as examples) suggested to us that the
condition
\begin{equation}\label{condition}\mathcal{R}_\Phi\subset {\rm\bf PrePer}(\Phi,\overline{K})\end{equation}
 is a necessary and sufficient condition for an endomorphism of $\SR$ to be finitely critical, where ${\rm\bf PrePer}(\Phi,\overline{K})$ denotes the set of preperiodic points for $\Phi$ defined over the algebraic closure of $K$. Our second result is the following one:
%

\begin{Teo}\label{iff}Let $\Phi$ be an endomorphism of $\SR$ of degree $\geq 2$ defined over a number field $K$. The map $\Phi$ is finitely critical if and only if the condition (\ref{condition}) holds.
\end{Teo}

In literature there exists already a name for the maps satisfying condition (\ref{condition}) that is \emph{post--critically finite} maps. We decide to conserve the name \emph{finitely critical} because a priori the set of post--critically finite maps contains the set of finitely critical maps, but a priori is not clear that also the other inclusion holds. Our Theorem \ref{iff} affirms that the two sets are equal. 
The request on the degree to be $\geq 2$ is due only to have critical points, otherwise the notion of finitely critical map has no meaning. The proof of Theorem \ref{iff}, that we will presente here, is an application of \cite[Theorem 1.6]{CanciPT}.

The present article is organized in several sections: In Section \ref{nd} we set the notation that we shall use throughout the paper; Section \ref{gn} is dedicated to Gauss norm and we shall present some remarks and results that we shall use in the proof of Theorem \ref{mainT}; Section \ref{intF} contains the generalization of Zannier's techniques as described before. In the last two sections we give the proofs of Theorem \ref{mainT} and Theorem \ref{iff} respectively.

\subsection*{Acknowledgment} I would like to thank Umberto Zannier for the introduction on the problem studied in Theorem \ref{mainT}, in particular for pointing out the literature that was the starting point of the present work (the article by Birch in \cite{Schneps} and his article \cite{Uz3}) and for the helpful discussions. I am grateful also to Luzien Szpiro for the discussions that motivated the result written in Theorem \ref{iff}. A part of this article was written at ICERM during the semester \emph{Complex and Arithmetic Dynamics} in 2012, I would like to thank ICERM and the organizers of the semester for the hospitality. I would like to thank also Pietro Corvaja for useful discussions. I'm grateful to Laura Paladino for reading the article and pointing out some inaccuracies.

\section{Notation}\label{nd}

Throughout all the paper $K$ will be a number field and $v$ a non archimedean valuation on $K$. We shall denote by $\mco$ its ring of $v$-integers, that is $\mco=\{P\in K\mid v(P)\geq 0\}$ and $\mco^*=\{P\in K\mid v(P)= 0\}$ its group of $v$--units. Sometime we shall need to extend the valuation $v$ to some finite extension of $K$. Therefore we assume that $v$ is extended to the whole $\overline{K}$. Every notion of good reduction considered in this article does not depend on the chosen extension of $v$.  

Any endomorphism $\Phi$ of $\SR$ defined over $K$ admits a $v$--normalized form, i.e. an expression of the form $\Phi(x)=F(x)/G(x)$ where $F,G\in \mco[x]$ are polynomials with no common factors and at least one coefficient of $F$ and $G$ is a $v$--unit. Such a $v$--normalized form always exists because $\mco$ is a principal ideal domain.  If $\Phi$ is written in $v$--normalized form, it is well defined its reduction modulo $v$, that we denote by $\Phi_v$, obtained by reducing modulo $v$ of the coefficients of $F$ and $G$. The expression $\Phi_v$ for a endomorphisms $\Phi$ of $\SR$ will always denote its reduction modulo $v$, whereas we will use the overline to denote the reduction modulo $v$ for polynomials and for elements in $K$. The condition for a map $\Phi$ to have simple good reduction at $v$ is equivalent to saying that the homogeneous resultant of $F,G$ is a $v$--unit. We use the adjective homogeneous because we consider the resultant of the homogeneous polynomials associated to $F$ and $G$. Indeed otherwise, if we consider the resultant of the (non homogeneous) polynomials $F$ and $G$ we can have some problems when the degrees of the polynomials $F$ and $G$ are different and the leading coefficient of the maximal degree polynomial is not $v$--invertible.  For example, with  $\Phi(x)=(2x^2+1)/x$ where $v$ is the valuation on $\mathbb{Q}$ associated to $2$, we have that the homogeneous resultant is 2 and the resultant of $2x^2+1$ and $x$ is 1. See \cite{Lang} for more details about the resultant of two polynomials.

 Recall that the Gauss valuation on the ring $\mco[x]$ (with respect to $x$ and $v$) is the function $\mco[x]\to \Z$
$$a^nx^n+a_{n-1}x^{n-1}+\ldots+a_0\mapsto \min\{v(a_n), v(a_{n-1}),\ldots,v(a_0)\}.$$
We extend in the usual way the Gauss valuation on the field $K(x)$. The Gauss valuation on $K(x)$ is discrete and the residue field is $K_0(\bar{x})$, where as explained in the introduction $\bar{x}$ denotes the reduction modulo $v$ of $x$. 

\section{Covers associated to extensions of Gauss norm}\label{gn}

In the proof of Theorem \ref{mainT}, it will be crucial the characterization of good reduction given in the introduction. We restate it in the following remark in order to have a precise reference in what follows:
\begin{Rem}\label{eqgr}
$\Phi$ has good reduction at $v$ if and only if the Gauss valuation on $K(\Phi)$ (with respect $\Phi$ and the valuation $v$) admits precisely one extension $w$ on $K(x)$, unramified and the extension of residue fields is separable and regular over $K_0$.
\end{Rem}
Zannier gave the proof of the statement in Remark \ref{eqgr} in \cite[Page 99]{Uz3}. In his proof we see that in the above statement the valuation $w$ is the Gauss valuation on $K(x)$, with respect a suitable transcendental element $y\in K(x)$ such that $K(x)=K(y)$.

The next two results about extensions of norms will be quite important in our proof of Theorem \ref{mainT}.

\vspace{5mm}
\noindent(THEOREM 9.14 in \cite{Jacobson}). Let $|\cdot |$ be an absolute value on the field $F$. Let $\hat{F}$ be the corresponding
completion of $F$, and let $E = F(u)$, where $u$ is algebraic over $F$ with minimum
polynomial $f(T)$ over $F$. Let $f_1(T),\ldots,f_h(T)$ be the distinct monic irreducible
factors of $f(T)$ in $\hat{F}[T]$. Then there are exactly $h$ extensions of $|\cdot|$ to absolute
values on $E$. The corresponding completions are isomorphic to the fields
$\hat{F}[T]/(f_j(T))$, $1 \leq j \leq h$, and the local degree are $n_i = \deg f_i(T)$.

\vspace{5mm}
\noindent (THEOREM 9.15. in \cite{Jacobson}) Let $F$ be a field with a non-archimedean absolute value
$|\cdot|$, $E$ an extension field of $F$ such that $[E:F] =n < \infty$, and let $|\cdot|_1,\ldots,|\cdot|_h$
be the extensions of $|\cdot|$ to absolute values on E. Let $e_i$ and $f_i$ be the ramification
index and residue degree of $E/F$ relative to $|\cdot|_i$ respectively. Then $\sum_{i=1}^h e_if_i \leq n$ and $\sum_{i=1}^h e_if_i = n$ 
if $|\cdot|$ is discrete and $E/F$ is separable.

\vspace{5mm}
Here we shall use the previous two results with $F=K(\Phi(x))$ with $\Phi(x)\in K(x)$ and $u=x$. If we denote $\Phi(x)=F(x)/G(x)$ with $F(x),G(x)\in K[x]$, then $x$ is a zero of the irreducible polynomial $f(T)\in K(\Phi)$ defined by $f(T)=F(T)-\Phi G(T)$. Therefore $n$ is equal to the degree of $\Phi$. The norm $|\cdot|$ over $F$  will be the Gauss norm (with respect $\Phi$ and $v$).

We shall use the following lemma which is a sort of generalization of Lemma 3.1 in \cite{Uz3}. 

\begin{Lemma}\label{e3.1}
Let $\Phi(x)\in K(x)$ be a non constant rational function. Suppose that  $w_1,\ldots,w_h$ are all the extensions of $v$ to $K(x)$ such that they induce the Gauss valuation on $K(\Phi)$ with respect to $\Phi$. Then there exists a finite extension $L$ of $K$, such that for each index $i\in\{1,\ldots h\}$ the valuation $w_i$ extends to a valuation over $L(x)$, that we still denote by  $w_i$, and there is an $y_i\in L(x)$ such that $L(x)=L(y_i)$ and $w_i$ is the Gauss valuation with respect $y_i$ and $v$. 
Furthermore the $y_i$'s can be chosen such that there exist an element $a\in L$ and an integer $m\geq 0$ such that for each $i\in\{1,\ldots,h\}$, there exist an integer $m_i\in\N$ and an integer $c_i\in \mco$ such that either $c_i=0$ or $v(c_i) < m_i$ and 
\beq\label{yi}a^mx=a^{m_i}y_i+c_i.\eeq \end{Lemma}
\begin{proof}
By \cite[Lemma 3.1]{Uz3}, for each index $i\in\{1,\ldots,h\}$ there exist a finite extension $K_i$ of $K$, an extension of $w_i$ to $K_i(x)$ and an element $\hat{y_i}\in K_i(x)$ such that $K_i(\hat{y_i})=K_i(x)$ and $w_i$ is the Gauss norm on $K_i(\hat{y_i})$ with respect to $\hat{y_i}$ and $v$. 
Choose as $L$ a finite extension of $K$ containing each field $K_i$ for all $i\in \{1,\ldots,h\}$. For each index $i$ we take the extension of $w_i$  over $L(x)$, that with abuse of notation we still denote by $w_i$. Therefore,  for all $i\in \{1,\ldots,h\}$, $w_i$ is the Gauss valuation of $L(x)$ with respect $\hat{y}_i$ and $v$.  Let  $a\in L$ be a uniformizer of $v$ over $L$. Let $m=\max\{-\min_{i\in \{1,\ldots,h\}}w_i( x),0\}$. By using the same Zannier's argument in \cite[Remark 3.5]{Uz3} we can choose the $\hat{y}_i$'s such that
\beq \label{a_ib_i}a^m x=a_i\hat{y}_i+b_i,\eeq
with $a_i, b_i\in L$. But the fact that $w_i(a^m x)\geq 0$ (by our choice of $m$) implies that $a_i, b_i$ are $v$--integers, since $w_i$ is the Gauss valuation with respect the element $\hat{y}_i$. For each $i\in\{1,\ldots,h\}$, there exists $m_i\geq 0$ such that  $v(a_i)=v(a)^{m_i}$. Hence there exists an $u_i\in\mco^*$ such that  $a_i=u_ia^{m_i}$. Suppose  $v(b_i)\geq m_i$, then take $y_i$ such that $\hat{y}_i=u_i^{-1}y_i-u^{-1}a^{-m_i}b_i$. Hence $a^m x=a^{m_i}y_i$. If $v(b_i)<m_i$ take just $\hat{y}_i=u_i^{-1}y_i$ and so we have that the $c_i$ in the statement of the lemma is equal to $b_i$. 
\end{proof}


\begin{Rem}\label{resdeg}{\upshape According with the notation used in Lemma \ref{e3.1} we have that $w_i$ is the Gauss norm on $L(y_i)$ with respect $y_i$. We consider the field extension $L(y_i)/L(\Phi_i(y_i))$ where $\Phi_i(y_i)=\Phi(x)=\Phi(a^{m_i}y_i+c_i)$, so $\Phi_i$ is again the morphisms $\Phi$ but viewed with respect the variable $y_i$. Hence, if $\Phi$ is written as $\Phi(x)=F(x)/G(x)$ with $F(x),G(x) \in \mco[x]$ coprime polynomials, we have that
$\Phi_i(y_i)=F(a^{m_i}y_i+c_i)/G(a^{m_i}y_i+c_i)$. Let  $\Phi_{i,v}(y_i)$ be the reduction modulo $v$ of $\Phi_i$ (so we are taking the reduction modulo $v$ of the $v$--normalized form of $\Phi_i$ with respect $y_i$). 
Let us denote $L_0$ the reduced field of $L$ with respect the valuation $v$. For each $i\in \{1,\ldots,h\}$, the reduced fields of $L(y_i)$ and $L(\Phi(y_i))$ with respect $w_i$ are $L_0(\overline{y}_i)$ and $L_0(\Phi_{i,v}(\overline{y}_i)$ respectively. Therefore the extension of reduced fields are $L_0(\overline{y}_i)/ L_0(\Phi_{i,v}(\overline{y}_i))$ of degree $[L_0(\overline{y}_i): L_0(\Phi_{i,v}(\overline{y}_i))]$, which is equal to the degree of the reduced map $\Phi_{i,v}(y_i)$.  In order to apply Theorem 9.14 and Theorem 9.15 in \cite{Jacobson}, we consider the same setting give before Lemma \ref{e3.1}. Let $f_1,\ldots,f_h$ be the irreducible factors of $f(T)=F(T)-\Phi G(T)$ in $\widehat{L(\Phi)}[T]$, where $\widehat{L(\Phi)}$ denotes the completion of $L(\Phi)$ with respect the Gauss norm. Let us denote $\widehat{L(x)}_i$ the completion of $L(x)$ with respect the Gauss norm $w_i$, for each $i\in\{1,\ldots,h\}$. Theorem 9.14 affirms that $\widehat{L(x)}_i$ is isomorphic to  $\widehat{L(\Phi)}[T]/f_i(T)$ and the local degree is the degree of $f_i$ that we denote by $n_i$. Note that we are in the case where the ramification index is one. Recall that the residue degree and the ramification index do not change passing to the completion. Thus we may compute the residue degree by considering the reduced fields. Since $w_i$ is the Gauss valuation with respect $y_i$ we have to write the elements in $L(x)$ as rational function in the variable $y_i$, for example as did before by rewriting $\Phi(x)$ as $\Phi_i(y_i)$.  The residue field of $\widehat{L(x)}_i$ is $L_0(\overline{y}_i)$ and the residue filed of $\widehat{L(\Phi)}$ is $L_0(\Phi_{i,v})$. Therefore the residue degree is equal to $\deg(\Phi_{i,v})$. 
By Proposition 9.3 in \cite{Jacobson} we have $n_i=\deg \Phi_{i,v}$. 
}
\end{Rem}

\section{Special integral form for a morphism}\label{intF}
Let $\Phi$ be an endomorphism of $\SR$ of degree $n\geq 2$ defined over $K$ verifying condition \ref{crv}. We assume that $p$ does not divide the order of the mondromy group or that there exist two distinct ramification values $\lambda$ and $\mu$ such that $p$ does not divide any nonzero  integer of the form $\left(\sum_{P\in A}e_\Phi(P)-\sum_{P\in B}e_\Phi(P)\right)$, where $A\subset \Phi^{-1}(\lambda), B\subset \Phi^{-1}(\mu)$. Let $n$ be the degree of $\Phi$. 
Our aim is to apply the Riemann-Hurwitz formula to the map $\Phi$ and the reduced maps $\Phi_{i,v}$, defined as in Remark \ref{resdeg}. We will show that it is possible to choose a particular model for $\Phi(x)=F(x)/G(x)$ such that the point at infinity is a ramification point in the fiber of the branch value 1, each other point in a ramification fiber is a  $v$--integers and the reduced modulo $v$ map $\Phi_v=\overline{F}/\overline{G}$ is not in $K_0(x^p)$, i.e. $\Phi_v$ is separable. Note that the existence of a model for $\Phi$ whose points in the ramification fibers are all $v$--integers is trivial; it is enough to take the polynomials $\alpha^n F(\alpha^{-1}x)$ and $\alpha^n G(\alpha^{-1}x)$  instead of $F(x)$ and $G(x)$ for an $\alpha\in K$ such that $v(\alpha)$ is big enough. The difficult part is to prove the existence of such an $\alpha$ for which the condition on separability holds. The fact that the points in the ramification fibers of $\Phi$ are $v$--integers will imply that the polynomials that we will take in consideration are monic with $v$--integral coefficients that is important because we will work with their reduction modulo $v$. But there is also another reason: The aim is to prove that the reduction modulo $v$ of the map $\Phi$, written as in the above form, has simple good reduction. Hence by  \cite[Theorem 1.6]{CanciPT} we know that  it has critically good reduction at $v$. Therefore, by \cite[Lemma 2.6]{CanciPT}, the condition $\Phi_v$ separable will imply that each point in the ramification fibers of $\Phi$ is a $v$--integer, because the point at infinity is a ramification point; furthermore the separability will imply also that there is only tame ramification. 
Note that an argument as the one contained in the proof of \cite[Lemma 3.6]{Uz3} proves that $p$ does not divide any $n_i$, thus the reduced field extension associated to any valuation $w_i$ is separable.

This section contains  the adaptations to our setting of Zannier's ideas contained in the proof in section 5 of \cite{Uz3}.  For the reader's convenience we present in a short way those ideas but we omit some details. 

We consider $K$ enlarged so that it contains all points of the ramification fibers of $\Phi$. Furthermore, we can assume that $K$ is so enlarged such that it has the same properties of the field $L$ described in the statement of Lemma \ref{e3.1}. As already remarked by Zannier in \cite[Section 2]{Uz3}, in testing good reduction we can replace $K$ with its completion. 
Let us denote by $\PGL_2(\mco)$ the subgroup of $\PGL_2(K)$ of the automorphisms associated to a matrix in GL$_2(\mco)$, that is the group of invertible matrices with coefficients in $\mco$, whose inverse still has coefficients in $\mco$. Recall that the action of $\PGL_2(\mco)$ is transitive on the set of triple of points in $\SR$ that remain (pairwise) distinct after reduction modulo $v$. Therefore, by condition (\ref{crv}), we can suppose that $\lambda$ and $\mu$ in the hypothesis of Theorem \ref{mainT} are equal to  $0$ and $\infty$, respectively. Indeed it is enough to consider $A\circ\Phi$ for a suitable $A\in\PGL_2(\mco)$. If $\mcr_\Phi$ has only two elements, then Theorem \ref{mainT} is trivially true, because $\Phi$ would be equivalent to the map $x^n$ that has bad reduction if and only if $p$ divides $n$. 
Therefore we may assume that the branch locus $\Phi(\mcr_\Phi)$ contains at least the three points $\{0,\infty,1\}$. Note that by our assumption on $\Phi$, any other branch point is a $v$--unit.  This fact will be useful in the proof of a lemma that is the generalization of \cite[Lemma 5.2]{Uz3}. Furthermore, without loss of generality, we can assume that $\infty\mapsto 1$ and $\infty$ is a ramification point (it is sufficient to take an element $A\in \PGL_2(\mco)$ that sends $\infty$ to a ramification point in the fiber of 1 and consider the map $\Phi\circ A$). In this way we have that $\deg(F)=\deg(G)$ and the leading coefficients of $F$ and $G$ are equal. We are going to  apply \cite[Proposition 4.1]{Uz3} that is an improvement of a result by Dwork and Robba about $p$--adic analytic continuations of Puiseux series (see \cite{DworkRobba}). As remarked before the field extensions associated to the valuations $w_i$'s are separable. This is important for the application of \cite[Proposition 4.1]{Uz3}.

Up to a translation of $x$ we may assume that one of the root of $F$ and $G$ is $0$. Let $\alpha\in K$ be such that 
$$v(\alpha)=-\min\left\{v(\beta)\mid \beta \in F^{-1}(0)\cup G^{-1}(0)\right\},$$
where we assume $v(\infty)=+\infty$.
By replacing $F(x), G(x)$ with $\alpha^n F(x/\alpha), \alpha^n G(x/\alpha)$ respectively, we may assume that the roots of $F,G$ are $v$--integers and at least one is a $v$--unit. In particular $F,G$ are monic in $\mco[x]$. Since $0\in F^{-1}(0)\cup G^{-1}(0)$,  we have that not all roots of $F$ and $G$ reduces modulo $v$ to the same point in $K_0$. This assumption will be crucial in the last part of this section.

We will see that this new model of $\Phi$ has the property that the fiber of a branch value contains only points in $\mco$ except the fiber over $1$ that contains $\infty$ and $v$--integers. In order to prove this, for an arbitrary branch value $\lambda$, different from $0$ and $\infty$, we consider 
the polynomial
\beq\label{eqlambda}
F(t)-\lambda G(t)=c_\lambda H_\lambda(t),\eeq
where $c_\lambda\in \mco$, $H_\lambda\in \mco[x]$ and at least one coefficient of $H_\lambda$ is a $v$--unit. 
As in \cite{Uz3} we are going to apply \cite[Proposition 4.1]{Uz3} to the polynomial 
$$f_\lambda(X,\Phi)\coloneqq \Phi F(X)-c_\lambda H_\lambda(X).$$

The assumptions concerning the field extensions in \cite[Proposition 4.1]{Uz3} are verified also for the polynomial $f_\lambda(X,\Phi)$, because $f_\lambda(X,\Phi)=(\Phi-1)F(X)+\lambda G(X)$. Therefore, we are considering the map $\lambda/(1-\Phi)$ instead of $\Phi$. That is the same of taking $A\circ \Phi$, with $A\in \PGL_2(K)$ given by the matrix 
$$\begin{pmatrix}0 &\lambda \\-1&1\end{pmatrix},$$
which has determinant equal to $\lambda$, that is a $v$--unit. 

In order to apply \cite[Proposition 4.1]{Uz3} it remains to prove that $f_\lambda(X,\Phi)$ does not have multiple roots in $X$ at  $z_0$, for each non zero element $z_0$ with $v(z_0)<1$. 
\begin{Lemma}
Let $z_0$ in $K^*$ such that the polynomial $f_\lambda(X,z_0)$ in the variable $X$ has a multiple root; then $v(z_0)=1$. 
\end{Lemma}
\begin{proof}Let $x_0$ be a multiple root of $f_\lambda(X,z_0)=z_0F(X)-H_\lambda(X)$. Since $F(X)$ and $H_\lambda(X)$ are coprime, we have that $F(x_0)\neq 0$. Then $z_0=H_\lambda(x_0)/F(x_0)$ and $z_0$ is a branch point of the map $\Psi_\lambda(x)=\lambda/(1-\Phi(x))$. Therefore we have that 
$$z_0\in \Psi_\lambda(R_{\Psi_\lambda})=\left\{\frac{\lambda}{1-\mu}\mid \mu\in \Phi(R_\Phi)\right\}.$$
Recall that 0 and 1 are branch values of $\Phi$ and that two different branch values have different reduction modulo $v$. Then $1-\mu$ is a $v$--unit for each  $\mu\in \Phi(R_\Phi)$, which implies that $\frac{\lambda}{1-\mu}$ is a $v$--unit. So $v(z_0)=1$.
\end{proof}
We consider the above polynomial with $\lambda=1$ and we apply below a slightly modification of the Zannier techniques in order to prove that $v(c_1)=0$. 

We consider the Puiseux expansions $\theta(z)$ of the algebraic function solutions of $f_1(\theta,z)$ around $z=0$. The point $\infty$ is in the fiber of $1$; denoting by $e$ the ramification degree of $\infty$ we have the same first family as defined in \cite[p. 107]{Uz3}:
$$\theta_i(z)=a_{-1}\zeta_e^iz^{-1/e}+a_0+a_1\zeta_e^{-i}z^{1/e}+\ldots , \quad i=0,1,\ldots, e-1$$
where $\zeta_e$ is a primitive $e$--th root of 1 and where $a_{-1}^e$ is the leading coefficient of $H_1$.

Furthermore, for any other root $b$ of $H_1$ with ramification degree $e_b$,
we have the family of solutions 
$$\theta_{b,i}(z)=b+b_1\zeta_{e_b} z^{1/e_b}+\ldots, \quad i=0,1,\ldots, e_b-1$$
where $\zeta_{e_b}$ is a primitive $e_b$--th root of the unity. Therefore all coefficients in the series $\theta_i$'s and $\theta_{b,i}$ are contained in a finite extension $L$ of $K$.
With the same arguments used in \cite[p. 108]{Uz3}, that use \cite[Prop. 4.1]{Uz3} and \cite[Prop .1.1, p.115]{DworkGS}, one proves that the above series $\theta_i$ and $\theta_{b,i}$ have $v$--integral coefficients. In particular this proves that the roots of $H$ are $v$--integers. 

We consider the factorization in $L[[z]]$ of the polynomial $z^mF(X)-H(X)$ , where $m$ is the lowest common multiple of $e$ and the $e_b$'s, exactly as done in \cite[from p. 108]{Uz3} with $e$ instead of $m$
\beq\label{facti}z^mF(X)-H(X)=\prod_{i=1}^e(z^{m/e}X-z^{m/e}\theta_i(z^m))\prod_{b\in H^{-1}(0)}\left(\prod_{i=1}^{e_b}(X-\theta_{i,b}(z^m)\right).\eeq
 
The proof goes on exactly as in \cite{Uz3} where if we assume that $v(c_1)>0$ we obtain that all the roots of $F,G,H$ must be congruent to a given one of them, that contradicts our assumption on $F$ and $G$. Therefore we have $v(c_1)=0$ and as explained in \cite[p. 109]{Uz3}, we deduce that the reductions $\overline{F}$ and $\overline{G}$ are linearly independent over $K(x^p)$, thus $\Phi_v=\overline{F}/\overline{G}$ is separable.

With the same above arguments given before, by considering a generic ramification value $\lambda$, we prove that each point in the fiber of $\lambda$ is $v$--integral.

\section{Proof of Theorem \ref{mainT}}\label{pmt}
Let $\Phi$ be an endomorphism of $\SR$ of degree $n\geq 2$ defined over $K$ verifying the properties as described in Section \ref{intF}. Therefore we are assuming that $p$ does not divide the order of the mondromy group or does not divide any nonzero  integer of the form $\left(\sum_{P\in A}e_\Phi(P)-\sum_{P\in B}e_\Phi(P)\right)$, where $A\subset \Phi^{-1}(0)$ and $ B\subset \Phi^{-1}(\infty)$. As remarked in the previous section, this assumption on $p$ implies that each residue field extension $K_0(\overline{y}_i)/K(\Phi_{i,v}(\overline{y}_i))$ is separable. This will be useful because we are going  to apply the Riemann--Hurwitz formula to each cover associated to the extension $K_0(\overline{y}_i)/K(\Phi_{i,v}(\overline{y}_i))$. 

Let the field $L$ and the extensions $w_i$ as defined in Lemma \ref{e3.1}. As did in Section \ref{intF}, we assume that $K$ is so enlarged so that $K=L$. 

By Lemma \ref{e3.1} we have that the valuation $w_i$'s over $K(x)$ are unramified as extensions of the Gauss valuation over $K(\Phi)$. Indeed, $w_i$ is the Gauss norm on $K(y_i)$ with respect $y_i$. Moreover the residue fields are regular over $K_0$. Therefore in order to prove that $\Phi$ has good reduction at $v$ it is enough to prove that $h=1$. Indeed by the arguments in Section \ref{intF} we have that  the extension of the residue fields is separable (see the characterization of good reduction given in Remark \ref{eqgr}).

As already seen in Remark \ref{resdeg}, for each $i\in \{1,\ldots,h\}$, the residue degree associated to $w_i$ is the degree of the reduced map $\Phi_{i,v}$ and it is equal to the local degree $n_i$. Since the extension are unramified, Theorem 9.15 in \cite{Jacobson} implies
\beq\label{degsum}\sum_{i=1}^{h}n_i=n.\eeq

The next statement represents a technical lemma, whose content is not so deep but it is useful to fix some notations that we will use in the rest of the proof.

\begin{Lemma}\label{Ord}
We use the same notation as in Lemma \ref{e3.1}. Let $i,j\in\{1,\ldots,h\}$. Let $a,m,m_i,m_j,c_i,c_j$ be as in Lemma \ref{e3.1}, that is $a^mx=a^{m_i}y_i+c_i=a^{m_j}y_j+c_j$. Suppose $m_i\leq m_j$. Then the following are equivalent:
\begin{itemize}
\item[i)]  there exists $\alpha\in\SR(K)$ such that $v(c_i-\alpha)\geq m_i$ and $v(c_j-\alpha)\geq m_j$;
\item[ii)] there exist an integer $n_{i,j}>0$ and an element $s_{i,j}\in \mco$ such that $y_i=a^{n_{i,j}}y_j+s_{i,j}$;
\item[iii)] there exists $\gamma\in\Phi^{-1}(\{0,\infty\})$ such that $v(c_i-a^m\gamma)\geq m_i$ and $v(c_j-a^m\gamma)\geq m_j$;
\end{itemize}
Furthermore, if $m_i=m_j$ and i), ii) and iii) hold, then $i=j$.
\end{Lemma}

\begin{proof} At first we prove {\it i)} $\Rightarrow$ {\it ii)}. The condition {\it i)} implies that $v(c_j-c_i)\geq m_i$. We have
$$y_i=a^{m_j-m_i}y_j+\dfrac{c_j-c_i}{a^{m_i}}.$$
Therefore it is sufficient to take $n_{i,j}=m_j-m_i$ and $s_{i,j}=(c_j-c_i)/a^{m_i}$.

Now we prove that $ii) \Rightarrow iii)$. Because of our assumption on $\Phi$, we have that the map $\Phi$ has the shape
$$\Phi(x)=\dfrac{\prod_{l=1}^n(x-\alpha_l)}{\prod_{k=1}^n(x-\beta_k)}$$
where $\alpha_l, \beta_k\in \mco$ for any indexes $l$ and $k$ (where in the numerator or in the denominator we could have some repeated factors).

The reduction of the $v$--normal form with respect the variable $y_i$ of the map 

$$\Phi(x)=\Phi(a^{m_j}y_j+c_j)= \dfrac{\prod_{l=1}^n(a^{m_j}y_j+c_j-a^m\alpha_l)}{\prod_{k=1}^n(a^{m_j}y_j+c_j-a^m\beta_k)}$$
must be transcendental over $K_0$, because $w_j$ is an extension of the Gauss norm of $K(\Phi)$. Note that if $v(c_j-a^m\alpha)=r<m_j$, then $a^{m_j}y_j+c_j-a^m\alpha=a^r(a^{m_j-r}y_j+(c_j-a^m\alpha)a^{-r})$, where $(a^{m_j-r}y_j+(c_j-a^m\alpha)a^{-r})\in\mco[y_j]$ and its reduction is in $K_0$.  Then there exists $\gamma\in\{\alpha_1,\ldots,\alpha_{n},\beta_1,\ldots,\beta_{n}\}$ such that $v(c_j-a^m\gamma)\geq m_j$. From condition ii) we deduce $c_j=a^{m_i}s_{i,j}+c_i$. Therefore the above argument prove that $v(a^{m_i}s_{i,j}+c_i-a^m\gamma)\geq m_j$ implying $v(c_i-a^m\gamma)\geq m_i$, because of our assumption $m_i\leq m_j$.

The implication $iii) \Rightarrow i)$ is completely trivial.

Now suppose that $m_i=m_j$. If {\it ii)} holds, then $n_{i,j}=0$. Hence $y_i=y_j+s_{i,j}$ and this is absurd because of the characterization of the Gauss norm, indeed if $y_i=y_j+s_{i,j}$, then $w_i$ and $w_j$ would be two different extensions of $v$ to $K(y_i)$ whose reduction of $y_i$ is transcendent.
\end{proof}

Lemma \ref{Ord} allows us to define a partial order $\leq$ on the set $\{w_1,\ldots, w_h\}$. Let $y_i$, $m_i$ and $c_i$ be defined as in Lemma \ref{e3.1} for all indexes $i$. For all $i, j$ we say that $w_i\leq w_j$ (or equivalently $y_i\leq y_j$) if $m_i\leq m_j$ and condition i), or the equivalent conditions ii) and iii), of Lemma \ref{Ord} holds. Actually condition ii) implies easily that $\leq$ is transitive and the case $m_i=m_j$ of Lemma \ref{Ord} assures that $\leq$ is antisymmetric; the reflexivity is trivially true. Therefore we can define a directed graph by using the above partial order in the canonical way. As usual we say that $w_j$ is a successor of $w_i$ if $w_i<w_j$ and there exists no $w_k$ with $w_i<w_k<w_j$. In the graph notation we say that the ordered pair $(w_i,w_j)$ is a directed edge or arrow of the graph.

Next lemma affirms that the above order admits a minimum. In graph theory notation we say that the graph is an arborescence. 
\begin{Lemma}\label{min}
Let $\Phi$ and $K$ be assumed as at the beginning of the present section.  Let $w_1,\ldots, w_h$ be the extensions over $K(x)$  of the Gauss valuation over $K(\Phi)$. The order described as above admit a minimum, i.e. there exists $i_0\in \{1,\ldots,h\}$ such that 
$w_{i_0}\leq w_i$ for all  $i\in \{1,\ldots,h\}$. Furthermore $w_{i_0}$ is the Gauss valuation with respect to $x$ and more precisely $y_{i_0}=x$.
\end{Lemma}
\begin{proof}Let the $y_i$'s, $m_i$'s and the $c_i$'s be as given in Lemma \ref{e3.1}Let $w$ be the Gauss valuation on $K(x)$. Since $\Phi(x)$ is written in the integral normal form as described in Section \ref{intF}, in particular the fact that the reduction $\Phi_v$ is not in $K_0$, we have that $w$ is an extension of the Gauss valuation on $K(\Phi)$. Therefore there exists an index $k\in\{1,\ldots,h\}$ such that $w=w_{k}$. We are going to prove that if there are three indexes $i_1,i_2,i_3$ such that 
\beq\label{to}w_{i_1}<w_{i_3},\quad w_{i_2}<w_{i_3},\eeq 
then we have either $w_{i_1}\leq w_{i_2}$ or  $w_{i_2}\leq w_{i_1}$.
We use the same notation as in Lemma \ref{e3.1}. By the above definition of ordering, according to Lemma \ref{Ord}, we have that the condition (\ref{to}) implies that $m_{i_1}<m_{i_3}$,  $m_{i_2}<m_{i_3}$ and there exists $\gamma_1,\gamma_2\in\SR(K)$ such that 
\beq\label{ci}v(c_{i_1}-\gamma_1)\geq m_{i_1},\ \ v(c_{i_3}-\gamma_1)\geq m_{i_3},\ \ v(c_{i_2}-\gamma_2)\geq m_{i_2},\ \ v(c_{i_3}-\gamma_2)\geq m_{i_3}.\eeq
Without loss of generality we may assume that $m_{i_1}\leq m_{i_2}$.
Thus we have to prove that there exists a $\gamma\in \SR(K)$ such that $v(c_{i_1}-\gamma)\geq m_{i_1}$ and $v(c_{i_2}-\gamma)\geq m_{i_2}$. We claim that $\gamma_2=\gamma$ has the previous property. From (\ref{ci}) we have
\beq\label{ic}v(c_{i_1}-c_{i_2}) =v(c_{i_1}-\gamma_1+\gamma_1-c_{i_3}+c_{i_3}-\gamma_2+\gamma_2-c_{i_2})\geq v(c_{i_1}-\gamma_1)= m_{i_1}\eeq 
Since $v(c_{i_2}-\gamma_2)\geq m_{i_2}$, it is enough to prove that $v(c_{i_1}-\gamma_2)\geq m_{i_1}$.
Let us suppose that $v(c_{i_1}-\gamma_2)<m_{i_1}$, then by (\ref{ic}) 
$$m_{i_1}\leq v(c_{i_1}-c_{i_2})=v(c_{i_1}-\gamma_2+\gamma_2-c_{i_2})=\min\{v(c_{i_1}-\gamma_2),v(\gamma_2-c_{i_2})\}=v(c_{i_1}-\gamma_2)<m_{i_1}$$
which is an absurd.

Since $w_k=w$ is the Gauss valuation with respect $x$ and also with respect to $y_k$, we have that $x=uy_k+t$ with some suitable $v$--unit $u$ and a $v$--integer $t$. Indeed, by Lemma \ref{e3.1} we have that $x=\frac{a^{m_k}}{a^m} y_k+\frac{t}{a^m}$. Hence if $m>m_k$ or $v(t)<m$ we have that $x$ reduces to the point $[1:0]$; if $m_k>m$ and $v(t)\geq m$, then $x$ reduces to a point in $K_0$. In both cases $w$ is not the Gauss valuation on $K(y_k)$, with respect $x$. Therefore we have $a^mx=a^muy_k+a^mt=a^{m_k}y_k+c_k$, where recall that either $c_k=0$ or $v(c_k)<m_k$. Since $y_k$ is a transcendental element over $K$, we have that $m=m_k$, $u=1$ and $c_k=a^mt$; but for our choice of $c_k$ we have that $c_k=0$, i.e. $x=y_k$. 

From our choice of the $v$--integral form we have that every point in the ramification fibers are $v$--integers (see last part in Section \ref{intF}), in particular over $0$ and $\infty$. Therefore, since $c_k=0$, we have that $v(c_k-a^m\delta)=v(a^m\delta)\geq m=m_k$ for each $\delta\in\Phi^{-1}(\{0,\infty\})$. Since $w_j$ is the extension of the Gauss norm on $K(\Phi)$ for all $j\in \{1,\ldots,h\}$, there exists a $\delta_j\in \Phi^{-1}(\{0,\infty\})$ such that $v(c_j-a^m \delta_j)\geq m_j$. So we have proven that either $w_j\geq w_k$ or $w_k\geq w_j$ for all  $j\in \{1,\ldots,h\}$. This proves that the graph given by the $w_i$'s is connected. Therefore, we deduce that there exists $i_0\in \{1,\ldots,h\}$ so that $w_{i_0}$ is the minimum of the order. Note that if the above order is not a total one, we see easily that $i_0=k$. But this holds in general, even if the order is total, indeed suppose that $w_{i_0}< w_k$. Thus $v(c_{i_0})<m_{i_0}<m_k$. We have $v\left(\frac{c_{i_0}-a^{m_k}\delta}{a^{m_{i_0}}}\right)<0$ for each $v$--integer $\delta$. Since every element of $\Phi^{-1}(\{0,\infty\})$ is a $v$--integer, the inequality $w_{i_0}< w_k$ would imply that $\Phi_{i_0,v}\in K_0$, that contradicts the choice of $y_{i_0}$.\end{proof}

By Lemma \ref{min}, up to renumbering the indexes, we may assume that $w_1$ is the minimum of the above order, so we have $x=y_1$. Let $m$ be the integer as in Lemma \ref{e3.1}, by Lemma \ref{Ord} we can assume that $m=m_1=0$.

For any index $i\in\{1,2,\ldots, h\}$ let us denote $U_i$ the subgraph of the vertices $w_j$ such that $w_j\geq w_i$. We shall call \emph{lenght of $U_i$} the maximal number $l$ such that there exists a chain $w_i=w_{i_0}<w_{i_1}<\ldots<w_{i_l}$. These objects are useful in the following application of the Riemann--Hurwitz Formula. Indeed, we want to evaluate the sum 

\beq\label{sRH}\sum_{i=1}^{h} \sum_{P\in \SR}\left(n_i-\#\{\Phi_{i,v}^{-1}(P)\}\right),\eeq
where recall that $n_i=\deg \Phi_{i,v}$.

As already remarked, our assumption about the characteristic $p$ implies that the extensions associated to the reduced maps $\Phi_{i,v}$ are separable. Thus, the Riemann--Hurwitz formula tells us that  $\sum_{P\in \SR}n_i-\#\{\Phi_{i,v}^{-1}(P)\}\leq 2n_i-2$. 
Therefore the sum in (\ref{sRH}) should be $\leq 2n-2h$. We are going to evaluate the sum in (\ref{sRH}), by evaluating each single $\#\Phi_{i,v}^{-1}(P)$. We shall prove that the sum in (\ref{sRH}) is  $\leq 2n-2h$ if and only if $h=1$, the map $\Phi$ does not have wild ramification and each ramification value of $\Phi_{v}$ is the reduction modulo $v$ of a ramification value of $\Phi$. 

In the next lines we shall define some technical objects useful in the remaining part of the present proof. For each branch point $\lambda$ (including $\infty$), we have 
$$\Phi_{i,v}^{-1}(\lambda)\subset \left\{\overline{\left(\frac{c_i-\gamma}{a^{m_i}}\right)}\mid \gamma\in \Phi^{-1}(\lambda)\setminus\{\infty\}, v(c_i-\gamma)\geq m_i\right\}\cup \{\infty\},$$ where $m_i$ and $c_i$ are the ones defined in Lemma \ref{e3.1}. 
When $\lambda\in \{0,\infty\}$, the above inclusion is completely clear because of
\beq\label{iv}\Phi_{i,v}(\overline{y}_i)=\frac{\overline{\prod_{l=1}^n(a^{m_i}y_i+c_i-\alpha_l)}}{\overline{\prod_{k=1}^n(a^{m_i}y_i+c_i-\beta_k)}}=\frac{A\prod_{k\in D}(\overline{y}_i+\overline{(c_i-\alpha_l)/a^{m_i}})}{B\prod_{k\in M}(\overline{y}_i+\overline{(c_i-\beta_r)/a^{m_i}})},\eeq
where $D=\{l\mid v(c_i-a^m\alpha_l)\geq m_i\}$, $M=\{k\mid v(c_i-a^m\beta_k)\geq m_i\}$ and $A$ and $B$ are the products of the reduction of the other factors after transformation of $\Phi_i(y_i)$ in a $v$--normalized form. Note that in the above fraction, some factors of the numerator and denominator of the non reduced map can simplify after reduction modulo $v$. Furthermore $\infty$ can be contained in $\Phi_{i,v}^{-1}(\lambda)$, e.g. if the degree of the two reduced polynomials of the numerator and denominator in (\ref{iv}) are different. 
For $\lambda\notin\{0,\infty\}$, it is enough to consider some composition on the left for some automorphism in $\PGL_2(\mco)$ that sends $\lambda$ to 0 and repeat the above arguments. 

For each index $i\in\{1,\ldots,h\}$, we denote by $R_i$ the set
$$R_i=\{\gamma\in\Phi^{-1}(\Phi( R_\Phi))\mid \gamma\neq \infty, v(c_i-\gamma)\geq m_i\}\cup \{\infty\}.$$

We shall give a technical definition (of counted point) that is useful in order to count the points in the fibers of the ${\Phi}_{i,v}$. More precisely it is useful to know, in a certain sense, for how many indexes $i\in\{1,\ldots,h\}$ the reduction modulo $v$ of an element in a fiber of $\Phi$ of a branch point is \emph{counted} as a point in a fiber of the maps ${\Phi}_{i,v}$'s. 

Let $\overline{\alpha}\in K_0$. For each $\mu\in \Phi(R_\Phi)$ consider the set 
$$A_{\mu, \overline{\alpha},i}=\{\gamma\in \Phi^{-1}(\mu)\cap R_i\mid \gamma\neq\infty,  \overline{(c_i-\gamma)/a^{m_i}}=\overline{\alpha}\}.$$


 For a fixed $\lambda\in \Phi(R_\Phi)$, consider the fiber $\Phi^{-1}(\lambda)$ and choose a labelling of its elements of the shape $\{\alpha_{\lambda,1},\ldots, \alpha_{\lambda,t}\}$. In this way we have fixed an order on this set, where $\alpha_{\lambda,i}\leq \alpha_{\lambda,j}$ if $i\leq j$. 
Moreover for an element $\gamma\in R_i$, let us denote by $l_i(\gamma)$ the number of indexes $k\in\{1,\ldots, h\}$ such that $w_i\leq w_k$ and $\gamma\in R_k$.  
With this notation we give the following definition.

\begin{Def}\label{defcount}
Let $\lambda\in \Phi(R_\Phi)$ and $\alpha_{\lambda,k}\in\{\alpha_{\lambda,1},\ldots, \alpha_{\lambda,t}\}=\Phi^{-1}(\lambda)$. Suppose that $\alpha_{\lambda,k}\in R_i$ and denote by  $\overline{\alpha}$ the reduction modulo $v$ of the $v$--integer $(c_i-\alpha_{\lambda,k})/a^{m_i}$, that is  $\alpha_{\lambda,k}\in A_{\lambda,\overline{\alpha}, i}$. We say that $\alpha_{\lambda,k}$ is \emph{counted in the fiber $\Phi_{i,v}^{-1}(\overline{\lambda})$} if the following conditions are verified:
\begin{enumerate}
\item $\# A_{\lambda,\overline{\alpha}, i}>\# A_{\mu,\overline{\alpha}, i}$ for all $\mu\in \Phi(R_\Phi)$ with $\mu\neq\lambda$;
\item $l_i(\alpha_{\lambda,k})= \min_{1\leq j\leq t}\{ l_j(\alpha_{\lambda,j})\}$;
\item $k$ is the minimum index such that the above condition 2. is verified.
\end{enumerate}
\end{Def}
Let us point out some remarks about this last definition. We show that the condition 1. in the above definition is necessary. Indeed, suppose that there exists a $\lambda$ different from $\mu$ such that $\# A_{\lambda,\overline{\alpha}, i}\leq\# A_{\mu,\overline{\alpha}, i}$. Up to taking $A\circ \Phi$ instead of $\Phi$, with $A$ a $v$--invertible automorphism in PGL$_2(\mco)$ that send $\lambda$ to 0 and $\mu$ to infinity, we have 
\beq\label{cl}\Phi_i(y_i)=\frac{f_i(y_i)\prod_{P\in  A_{\lambda,\overline{\alpha}, i} }\left(y_i-\frac{c_i-P}{a^{m_i}}\right)}{g_i(y_i)\prod_{Q\in  A_{\mu,\overline{\alpha}, i} }\left(y_i-\frac{c_i-Q}{a^{m_i}}\right)},\eeq
where the polynomials $f_i$ and $g_i$ are not divisible by any linear factors appearing in the products. Hence in the reduction $\Phi_{i,v}$ the product in the numerator in (\ref{cl}) disappears in the reduction modulo $v$, because there is cancellation with the product in the denominator (or a part of it). Therefore in this sense we can not say that an element in $A_{\lambda,\overline{\alpha}, i}$ is counted in the preimage of  $\lambda$ for the reduced map $\Phi_{i,v}$.

Note that for each $\overline{\alpha}\in \Phi^{-1}_{i,v}(\overline{\lambda})$, there exists a unique $\alpha_{\lambda,k}\in A_{\lambda,\overline{\alpha},i}$ counted in the fiber $\Phi^{-1}_{i,v}(\overline{\lambda})$.  Indeed Definition \ref{defcount} defines a correspondence from the set of the pairs $(\overline{\alpha},\Phi_{i,v})$  to the set $\Phi^{-1}(\Phi(R_\Phi))$. The condition 2. is given in order to obtain a correspondence, maybe not injective, but that looks like a injective correspondence as much as possible. Indeed first of all note that $R_j\subset R_i$ for each index $j$ such that $w_i<w_j$. Furthermore if $\alpha,\beta\in R_j$, then 
\beq\label{uni}v\left(\frac{c_i-\alpha}{a^{m_i}}-\frac{c_i-\beta}{a^{m_i}}\right)=v\left(\frac{c_j-\alpha}{a^{m_i}}-\frac{c_j-\beta}{a^{m_i}}\right)\geq m_j-m_i>0,\eeq
since $m_i<m_j$. Thus all points of $R_j$ reduce to the same point if they are considered as points in $R_i$. More concretely,  we have that there exists an $\overline{\alpha}\in K_0$ such that $\overline{\frac{c_j-\gamma}{a^{m_i}}}=\overline{\alpha}$ holds  for all elements $\gamma\in R_j$. For example the conditions 2. says that if there exist a unique $\beta\in R_i\setminus R_j$ such that $\overline{\frac{c_j-\beta}{a^{m_i}}}=\overline{\alpha}$ and condition 1. is verified, then in the Riemann--Hurwitz formula for the map $\Phi_{i,v}$ we \emph{choose} $\beta$ as a representative of the class $\overline{\alpha}$ in the fiber of $\Phi_{i,v}$ over $\overline{\lambda}$, instead of taking one of the elements in $R_j$. Each other $\gamma$ in $R_j$ such that $\overline{\frac{c_j-\gamma}{a^{m_i}}}=\overline{\alpha}$ could be counted by considering one of the maps $\Phi_{k,v}$ with $w_k\geq w_j$. Actually, the condition 3. is given only because we could have more than one index $k$ verifying the first two conditions.

Now we are ready to state a technical lemma useful in evaluating the sum in (\ref{sRH}). 

\begin{Lemma}\label{count}
Let $\Phi$, $K$, and the $R_i$'s be as above. For an arbitrary fixed $i\in \{1,\ldots, h\}$, let $i_1,\ldots, i_m\in \{1,\ldots, h\}$ be the set of indexes of all extensions $w_{i_s}$ such that $w_{i_s}>w_i$. Let us set $i=i_0$.
Then
\beq\label{spRH}\sum_{s=0}^{m} \sum_{\lambda\in \Phi(R_\Phi)}\#\{{\Phi}_{i_s}^{-1}(\overline{\lambda})\}\leq |R_{i_0}|+2m.\eeq
Furthermore if the above inequality is an equality, then all the following properties are verified: 
\begin{itemize}
\item[i)] for each $\gamma\in R_{i_0}$ there exists an index $i_s\in\{i_0, i_1,\ldots, i_m\}$ such that $\gamma$ is counted in a fiber of the shape ${\Phi}_{i_s,v}^{-1}(\lambda)$ for a ramification value $\lambda$ for $\Phi$;
\item[ii)] up to a permutation of the indexes, we can suppose that $\{i_1,\ldots, i_t\}$ is the full set of successors of $i_0$. For each $i_k\in \{i_1,\ldots, i_t\}$ there exists a (unique) $\gamma\in R_{i_0}\setminus\{\infty\}$ such that $\gamma$ is counted in ${\Phi}_{i_0,v}^{-1}(\overline{\lambda})$ and in ${\Phi}_{i_k,v}^{-1}(\overline{\lambda})$ for a ramification value $\lambda$ for $\Phi$;
\item[iii)]  for each $s\in\{0,\ldots,m\}$, the point at infinity is in a fiber of the shape $\Phi_{i_s,v}^{-1}(\overline{\lambda})$ for a ramification value $\lambda$ for $\Phi$. \end{itemize}
\end{Lemma}
\begin{proof}
We prove the lemma by induction on the length $l$ of $U_{i_0}$.

Suppose $l=0$, that means that $w_{i_0}$ is a maximal element in the graph. Therefore $\sum_{P\in \Phi(R_\Phi)}\#\{{\Phi}_{{i_0},v}^{-1}(\bar{P})\}\leq |R_{i_0}|$ and the equality is verified if and only if i) and iii) hold, since ii) is trivially verified. 

Now we suppose the statement true when the length is $< l$ and we shall prove that it is true for $U_i$ of length $l$. Up to renumbering the indexes, we can suppose that $\{w_{i_1},\ldots, w_{i_t}\}$ is the set of all successors of  $w_{i_0}$. 
By induction we have 
\beq\label{R_i}\sum_{k=0}^{m} \sum_{\lambda\in \Phi(R_\Phi)}\#\{{\Phi}_{i_k,v}^{-1}(\overline{\lambda})\}\leq \sum_{\lambda\in \Phi(R_\Phi)}\#\{{\Phi}_{i_0,v}^{-1}(\overline{\lambda})\}+2(m-t)+\sum_{k=1}^{t}|R_{i_k}|.\eeq
Let us set $S_{i_0}= R_{i_0}\setminus(R_{i_1}\cup\ldots \cup R_{i_t})$. Therefore the following union is disjoint:
\beq\label{S_0}R_{i_0}=\{\infty \}\cup \left(R_{i_1}\setminus \{\infty\}\right)\cup\ldots \cup \left(R_{i_t}\setminus \{\infty\}\right)\cup S_{i_0}.\eeq
Then
\beq\label{cR_i}|R_{i_0}|=1+\sum_{k=1}^t|R_{i_k}| -t+|S_{i_0}|.\eeq
Note that by applying the inequality in (\ref{uni}) we see that by (\ref{S_0}) we have
\beq\label{Si0}\sum_{\lambda\in \Phi(R_\Phi)}\#\{{\Phi}_{i_0,v}^{-1}(\overline{\lambda})\}\leq |S_{i_0}|+1+t.\eeq
Putting this last inequality in (\ref{R_i}) we obtain (\ref{spRH}), since (\ref{cR_i}) holds.


Suppose now that the inequality in (\ref{spRH}) is an equality, then (\ref{R_i}) and (\ref{Si0}) are equalities. Condition i) is verified, indeed if $\gamma \in S_{i_0}\cup\{\infty\}$ consider the equality in (\ref{Si0}). If $\gamma\in R_{i_0}\setminus S_{i_0}$ consider (\ref{R_i}) and the inductive hypothesis. Condition ii) holds by equality (\ref{Si0}). Condition iii) is verified by (\ref{Si0}) for $s=0$ and by (\ref{R_i}) and the inductive hypothesis for any other $s\neq 0$. \end{proof}

\begin{Lemma}\label{grad}
Let $\Phi$, $K$, and the $R_i$'s be as above. For an arbitrary fixed $i\in \{1,\ldots, h\}$, let $i_1,\ldots, i_m\in \{1,\ldots, h\}$ be the set of indexes of all extensions $w_{i_s}$ such that $w_{i_s}>w_i$. Let us set $i=i_0$. Let $A_i\coloneqq R_i\cap \Phi^{-1}(0)\setminus\{\infty\}$ and $B_i\coloneqq R_i\cap \Phi^{-1}(\infty)\setminus\{\infty\}$. Let $a_i\coloneqq \sum_{\alpha\in A_i}e_\Phi(\alpha)$ and similarly $b_i\coloneqq \sum_{\beta\in B_i}e_\Phi(\beta)$. Suppose that in (\ref{spRH}) the equality holds, then 
\beq\label{cdeg} \sum_{k=0}^{m}\deg({\Phi}_{i_k,v})\geq \max\{a_{i_0},b_{i_0}\}.\eeq
If  at least one element $\gamma$, of the type as described in ii) of the Lemma \ref{count}, is in $A_{i_0}\cup B_{i_0}$, then (\ref{cdeg}) is a strict inequality.
\end{Lemma}
\begin{proof}
We prove the inequality (\ref{cdeg}) by induction on the length of $U_{i_0}$. Since the equality in (\ref{spRH}) holds, then the properties i), ii) and iii) of Lemma \ref{count} are verified. 
Suppose that the length of $U_{i_0}$ is zero. Therefore each element of $R_{i_0}$ is counted in ${\Phi}_{i,v}$. That implies that any two different elements $\gamma_1,\gamma_2\in R_{i_0}$ reduce to two different elements
$\overline{\frac{c_i-\gamma_1}{a^{m_i}}}, \overline{\frac{c_i-\gamma_1}{a^{m_i}}}$ in $K_0$. Therefore we have that 
$${\Phi}_{i,v}(y_i)=\frac{\prod_{\alpha\in A_i}\left(y_i- \overline{\frac{c_i-\alpha}{a^{m_i}}}\right)^{e_\Phi(\alpha)}}{\prod_{\beta\in B_i}\left(y_i- \overline{\frac{c_i-\beta}{a^{m_i}}}\right)^{e_\Phi(\beta)}},$$
where there is no cancellation between the numerator and the denominator. Therefore (\ref{cdeg}) holds, more precisely it is an equality.

We suppose the inequality (\ref{cdeg}) is true when the length of a $U_{i_0}$ is $< m$ and we shall prove that it is true for $U_{i_0}$ of length $m$. Up to renumbering the indexes, we can suppose that $\{w_{i_1},\ldots, w_{i_t}\}$ is the set of all successors of  $w_i$. Note that if the equality in (\ref{spRH}) holds, then similar equalities hold for each successor $w_{i_k}$ of $w_{i_0}$.

By the inductive hypothesis we have 
$$\sum_{k=0}^{m}\deg({\Phi}_{i_k,v})=\deg({\Phi}_{i_0,v})+\sum_{k=1}^t\sum_{w_j\geq w_{i_k}} \deg({\Phi}_{j,v})\geq \deg({\Phi}_{i_0,v})+\sum_{k=1}^t\max\{a_{i_k},b_{i_k}\}.$$
Therefore it is enough to prove the following inequality
\beq\label{C}\deg({\Phi}_{i_0,v})+\sum_{k=1}^{t}\max\{a_{i_k},b_{i_k}\}\geq \max\{a_{i_0},b_{i_0}\},\eeq
that is implied by the equality 
\beq\label{id0}\deg({\Phi}_{{i_0},v})+\sum_{k=1}^{t}\min\{a_{i_k},b_{i_k}\}=\max\{a_{i_0},b_{i_0}\}.\eeq
Hence, in  order to prove (\ref{cdeg}) it is sufficient to prove the equality (\ref{id0}). Since in (\ref{spRH}) we have the equality, then every element in $A_{i_0}\setminus (A_{i_1}\cup \ldots\cup A_{i_t})$ and  $B_{i_0}\setminus (B_{i_1}\cup \ldots\cup B_{i_t})$ is counted. Therefore $\deg({\Phi}_{{i_0},v})$ is equal to $\max\{a_{i_0},b_{i_0}\}$ minus the number of cancellations in the numerator and denominator of the elements in $A_{i_k}$ and $B_{i_k}$ for each $k\in\{1,\ldots, h\}$. Recall that for each fixed index $j$ such that $w_j>w_{i_0}$, all points of the shape $\frac{c_{i_0}-\gamma}{a^{m_{i_0}}}$ with $\gamma\in R_j$ collide to the same point modulo $v$. Thus we have 
$$\deg({\Phi}_{{i_0},v})=\max\{a_{i_0},b_{i_0}\}-\sum_{k=1}^{t}\min\{a_{i_k},b_{i_k}\},$$
that proves (\ref{id0}).

Now suppose that there exists one element $\gamma$ as described in ii) of the Lemma \ref{count} that belongs to $A_{i_0}\cup B_{i_0}$. Thus, $\gamma$ is counted in ${\Phi}_{i_0,v}^{-1}(\overline{\lambda})$ and in ${\Phi}_{i_k,v}^{-1}(\overline{\lambda})$ for a suitable successor $w_{i_k}$ of $w_{i_0}$ and $\overline{\lambda}\in\{0,\infty\}$. In particular, this means that $\gamma$ is in $ R_{i_k}$ and is counted in ${\Phi_{{i_k},v}}^{-1}(\overline{\lambda})$. Therefore $\max\{a_{i_k},b_{i_k}\}-\min \{a_{i_k},b_{i_k}\}>0$. Hence the inequality in (\ref{C}) is strict and the one in (\ref{cdeg}) too.
\end{proof}

The final part of the proof is an application of Lemma \ref{count} and Lemma \ref{grad} with ${i_0}=1$. Note that $R_1=\Phi^{-1}(\Phi(R_\Phi))$. We have
\begin{align}2n-2h&\geq\sum_{k=1}^{h} \sum_{\gamma\in \SR(K_0)}\left(n_i-\#\{{\Phi}_{i_k,v}^{-1}(\gamma)\}\right)\label{tame} \\&\geq\sum_{k=1}^{h} \sum_{P\in \Phi(R_\Phi)}\left(n_i-\#\{{\Phi}_{i_k,v}^{-1}(\bar{P})\}\right)\label{extra}\\
&= \sum_{P\in \Phi(R_\Phi)}\sum_{k=1}^{h}\left(n_i-\#\{{\Phi}_{i_k,v}^{-1}(\bar{P})\}\right)\nonumber\\
&= \sum_{P\in \Phi(R_\Phi)}n-\sum_{P\in \Phi(R_\Phi)}\sum_{k=1}^{h}\#\{{\Phi}_{i_k,v}^{-1}(\bar{P})\}\nonumber\\
&\geq n|\Phi(R_\Phi)|-|R_1|-2(h-1)\label{id}\\
&= \sum_{P\in \SR(K)}\left(n-\#\{{\Phi}^{-1}({P})\}\right)-2(h-1)\nonumber\\
&= 2n-2h.\nonumber\end{align}
In (\ref{id}) we have applied Lemma \ref{count}. It is clear that the inequality in (\ref{tame}), (\ref{extra}) and  (\ref{id}) have to be identities. Therefore, we have that the properties i), ii) and iii) of Lemma \ref{count} hold. Suppose $h>1$; up to considering a change $A\circ\Phi$ for a suitable $A\in\PGL_2(\mco)$, we can suppose that  $\gamma\in A_1\cup B_1$ for one of the $\gamma$ as in ii). Therefore by Lemma \ref{grad} we have that 
$$\sum_{k=1}^{h}\deg({\Phi}_{i_k,v})=\sum_{k=1}^{h}n_k> \max\{a_1,b_1\}=n,$$
that contradicts (\ref{degsum}). This proves $h=1$.

\section{Critically good reduction}\label{cgr}
The proof of the Theorem \ref{iff} in an application of the following two results. 
\begin{Teo}[\cite{Silverman}]\label{tsil}
Let $K$ be a number field and $v$ a non archimedean
valuation. Let $\Phi$ and $\Psi$ be two endomorphisms of $\SR$
with simply good reduction at $v$. Denote by $\Phi_v$ and $\Psi_v$ the
reductions modulo $v$ of $\Phi$ and $\Psi$ respectively. Then the
composition $\Phi\circ\Psi$ has good reduction at $v$ and its
reduction $(\Phi\circ\Psi)_v$ is such that
$$ (\Phi\circ\Psi)_v=\Phi_v\circ\Psi_v.$$
\end{Teo}

This theorem implies that if $\Phi$ and $\Psi$ are as in Theorem \ref{tsil} and $\Phi_v$ and $\Psi_v$ are separable maps, then also the reduction of the composition $(\Phi\circ\Psi)_v$ is separable. The second tool for the proof of Theorem \ref{iff} is the following theorem.

\begin{Teo}[\cite{CanciPT}]\label{CPT}
Let $\Phi:\SR\to\SR$ be a morphism defined over a number field $K$.
Let $v$ be a finite place of $K$. Let $\Phi_v$ be the reduction modulo $v$ of $\Phi$. Let us suppose that  $\Phi_v$ is separable. Then  the following are equivalent:
\begin{itemize}
\item[a)]   $\Phi$ is C.G.R. at $v$;
\item[b)]   $\Phi$ is S.G.R. at $v$ and $\#\Phi(\mathcal{R}_{\Phi})=\#(\Phi(\mathcal{R}_{\Phi}))_v$.
\end{itemize}
\end{Teo}

\begin{proof}[Proof of Theorem \ref{iff}] 
Recall that the notions of simple good reduction and critically good reduction are preserved by taking finite extensions of $K$. Thus, we may assume that $\SR(K)$ contains all ramification points of $\Phi$. \\
Suppose that $\Phi$ is a finitely critical map. Let $\Phi^n$ be the $n$--th iterated map of $\Phi$. Denote by $\mathcal{B}_n$ the branch locus of $\Phi^n$, that is
$$\mathcal{B}_n=\Phi^n(\mathcal{R}_{\Phi^n})=\bigcup_{i=1}^n\Phi^i(\mathcal{R}_\Phi).$$
If $\mathcal{R}_\Phi$ contains a $K$--rational point not in
${\rm\bf PrePer}(\Phi,K)$, then the cardinality of
$\mathcal{B}_n(K)$ tends to $\infty$  with $n\to\infty$. Thus, the set of
valuations of $K$ of bad critically reduction of $\Phi^n$ grows up
with $n\to\infty$, in contradiction with the hypothesis that $\Phi$ is finitely critical.

Now we prove that the condition
(\ref{condition}) implies that $\Phi$ is finitely critical.
Thus we assume that  (\ref{condition}) holds. Then it follows that the set of postcritically points 
$${\rm PostCrit}_\Phi=\bigcup_{n\geq 1}\mathcal{B}_n$$
is a finite set. Therefore, there exists a finite set $S$ of
valuations of $K$ containing all the ones of bad simple reduction
of $\Phi$, all the valuations $w$ such that at least two different
points of ${\rm PostCrit}_\Phi$ collide modulo $w$ and all the valuations of $K$ that
are over a prime number less or equal the degree of $\Phi$. In
this way
we have that $\Phi$ has simple and critically good reduction at any valuation $v\notin S$ and the reduction map $\Phi_v$ is separable. \\
From Theorem \ref{tsil} we deduce that, for any
$n\geq 1$, the $n$--th iterate $\Phi^n$ has simple good reduction and the reduction $(\Phi^n)_v$ is separable
at any $v\notin S$. Furthermore by
$$\Phi^n(\mathcal{R}_{\Phi^n})=\bigcup_{i=1}^n\Phi^i(\mathcal{R}_\Phi)\subseteq {\rm PostCrit}_\Phi,$$
we have that $\#\Phi^n(\mathcal{R}_{\Phi^n})=\#(\Phi^n(\mathcal{R}_{\Phi^n}))_v$ for any $v\notin S$. Thus, by Theorem \ref{CPT}, we have that $\Phi^n$ has critical good reduction at any $v\notin S$ for all $n\in\N$.
\end{proof}

As an application of Theorem \ref{iff} we present the following two families of remarkable examples of rational maps.

\subsection{Quadratic polynomials} 
Let $K$ be a field that does not have characteristic $2$. We are considering endomorphisms of $\SR$ associated to quadratic polynomials of the form
\beq\label{qpol}\phi(x)=Ax^2+Bx+C\eeq
with $A,B,C\in K$. As remarked by Silverman in \cite[Section 4.2.1 p.156]{Silverman} each polynomial as in (\ref{qpol}) is conjugate via an element of ${\rm PGL}_2(K)$ to a polynomial of the form $x^2+c$ with $c\in K$.

Now we study the case $K=\Q$. As a corollary of Theorem \ref{iff} we have the following result.
\begin{Cor} The unique quadratic polynomials in $\Q[x]$ that are finitely critical maps are those conjugated to one of the following polynomials:
$$ x^2,\quad x^2-1, \quad x^2-2.$$
\end{Cor}
\begin{proof}As remarked before, any quadratic polynomial in $\Q[x]$ is conjugated to a polynomial of the type $\phi_c(x)=x^2+c$ with $c\in\Q$. For any $c\in\Q$ the map $\phi_c(x)$ ramifies at points $\infty$ and $0$. The point $\infty$ is clearly a fixed point. Hence, by Theorem \ref{iff} we have to verify whether $0$ is a preperiodic point for $\phi_c$. Note that $0$ is a preperiodic point for $\phi_c$ if and only if $c$ is a preperiodic point for $\phi_c$ too. Note that since $\phi_c$ is a monic polynomial, then for any $n,m\in\N$ and $n>0$, we have that if $\phi_c^{n+m}(c)= \phi^m_c(c)$, then $c$ is a zero of a monic polynomial with integral coefficients. Hence $c\in\Z$. Now it is clear that for any positive integer $c$, the point $0$ is not preperiodic for $\phi_c$; the same holds for any integer $c<-2$. On the contrary $0$ is a preperiodic point for all $c\in\{0,-1,-2\}$.
\end{proof}

\subsection{Latt\`es Maps}
Recall the definition of Latt\`es map:
\begin{Def}\label{Lmap}
A \emph{Latt\`es Map} is a map $\Phi\colon\SR\to\SR$ of degree $>1$ that fits into a 
commutative diagram
\begin{center}
\begin{equation}\label{diagram}
\begin{diagram}
\node{\ E}\arrow{e,t}{\psi} \arrow{s,l}{\pi}
\node{\ E}\arrow{s,r}{\pi}\\
\node{\ \SR}\arrow{e,t}{\Phi}\node{\SR}
\end{diagram}
\end{equation}

\end{center}
in which $E$ is an elliptic curve, the map $\psi$ is an endomorphism of $E$, and $\pi$ is a finite separable covering. The endomorphism $\Phi$ is called \emph{Flexible} Latt\`es Map, if the maps $\psi$ and $\pi$ have the following extra conditions: $\psi(P) = [m](P) +T$ for some $m\in\Z$ and some $T\in E$ and $\pi$ satisfies $\deg(\pi) = 2$ and $\pi(P ) = \pi(- P)$ for all $P\in E$.
\end{Def}

With some very insignificant modifications of Szpiro and Tucker's arguments in \cite{SzT} it is possible to see that any flexible latt\`es map is finitely critical. We resume briefly the Szpiro Tucker's argument. 

Let $K$ be a fixed number field. Let $m\geq 2$ be an integer. Let $E$ be an elliptic curve defined over $K$. Let $S$ be a finite set of valuations of $K$ that includes at least each prime lying over each prime dividing $2m$ and all primes of bad reduction for $E$. Take $S$ enlarged so that the ring of $S$--integers of $K$ is a P.I.D.. Take a planar model $y^2=F(x)$ where $F$ is a polynomial defined by $S$--integral coefficients with leading coefficient an $S$--unit (such a model exists as noted by Serre in \cite{Serre}). Let $T$ be a $K$--rational element in $E[2]$ (as usual $E[n]$ denotes the set of $n$--torsion points of $E$ defined over $\overline{\Q}$ for every positive integer $n$). Let $\psi\colon E\to E$ defined by $\psi(P)=mP+T$. Denote by $\Phi$ the endomorphism of $\SR$ which makes the diagram (\ref{diagram}) commutative, 
where $\pi$ in the diagram denotes the double cover $E\to\SR$ which sends $(x,y)\mapsto x$.
Following the Szpiro and Tucker's arguments in \cite[pages 719 and 720]{SzT} it is possible to see that $\Phi$ and all its iterates have critically good reduction at any valuation outside $S$: for each positive integer $n$, since $\psi^n$ is an endomorphism of a curves of genus 1, it is \'etale. From this it is possible to see that $\mcr_{\phi^n}$ and $\phi^n(\mcr_{\phi^n})$ are contained in the set of the $x$--coordinate of the points of $E[2m^n]$. From the choice of $S$, we know that the reduction modulo $v$ of the points in $E[2m^n]$ is injective (we mean that $P,Q\in E[2m^n]$ $P\neq Q$ if and only if $P_v\neq Q_v$, where $P_v$ and $Q_v$ denotes the reduction modulo $v$ of the points $P$ and $Q$ respectively). To see this injectivity it is sufficient to apply \cite[Proposition VII.3.1]{Silverman2}. From the planar model $y^2=F(x)$, we see that the map $\Phi^n$ has critical good reduction for all $v\notin S$. To extend these above arguments to the entire family of Flexible Latt\`es maps it is sufficient to apply \cite[Proposition 6.51]{Silverman}. 

As an application of Theorem \ref{iff} we have the following more general result.
\begin{Cor}
Latt\`es maps are finitely critical.
\end{Cor}
\begin{proof} Let $\Phi$ be a Latt\`es map associated to an elliptic curve $E$ and a finite cover $\pi$ as in Definition \ref{Lmap}. 
From Proposition 6.45 in \cite{Silverman} we have that the set of postcritically points $ {\rm PostCrit}_\Phi$ is exactly the set $\pi(\mcr_\pi)$. Hence ${\rm PostCrit}_\Phi$ is a finite set and so condition  (\ref{condition})  holds. 
\end{proof}

\end{document}